\documentclass[english]{article}
\usepackage[latin9]{inputenc}
\usepackage{amsmath}
\usepackage{amssymb}
\usepackage{esint}

\makeatletter

\title{Newtheorem and theoremstyle test}
\author{Michael Downes\\updated by Barbara Beeton}

\usepackage[exercise]{amsthm}

\newtheorem{thm}{Theorem}[section]

\theoremstyle{plain}
\newtheorem{letterthm}{Theorem}
\theoremstyle{plain}
\newtheorem{mainthm}{Theorem}

\newtheorem{lem}[thm]{Lemma}

\theoremstyle{remark}

\theoremstyle{plain}
\newtheorem{Def}{Definition}

\newtheoremstyle{note}
{3pt}
{3pt}
{}
{}
{\itshape}
{:}
{.5em}
{}

\theoremstyle{note}

\newtheoremstyle{citing}
{3pt}
{3pt}
{\itshape}
{}
{\bfseries}
{.}
{.5em}
{\thmnote{#3}}

\theoremstyle{citing}

\newtheoremstyle{break}
{9pt}
{9pt}
{\itshape}
{}
{\bfseries}
{.}
{\newline}
{}

\theoremstyle{break}

\theoremstyle{exercise}

\swapnumbers
\theoremstyle{plain}

\let\lvert=|\let\rvert=|

\usepackage{babel}

\usepackage{babel}

\usepackage{babel}

\usepackage{babel}

\makeatother

\usepackage{babel}
\begin{document}
	
	\title{Isolated singularities of certain quasi-linear elliptic inequalities
		II}
	
	\author{Shiguang Ma and Shengyang Zang}
	\maketitle
	\begin{abstract}
		
		In this paper, we study the asymptotic behavior of isolated singularities for quasilinear elliptic equations of $p$-Laplacian type. We focus on the $p$-Laplace equation with the critical Serrin exponent $n(p-1)/(n-p)$ and the $n$-Laplace equation with critical exponential growth. We prove that any positive solution either admits a $C^{1,\alpha}$ extension to the origin or exhibits a precise singular asymptotic behavior. These results provide a complement to our previous study on $p$-Laplacian type inequalities.
	\end{abstract}
	
	\section{Introduction}
	The characterization of the behavior of isolated singularities for partial differential equations of the form
	\[
	-\Delta_{p}u=f(x,u,\nabla u)
	\]
	has attracted considerable attention over the past decades with extensive literature on this topic; see, e.g., \cite{Se1964,Se1965,Lio1980,BL1981,GS1981,Av1987,GV1988,CGS1989,BV1989,BP2001,SZ2002,PQS2007,PV2008} and the references therein. For the study of isolated singularities of $p$-superharmonic functions and applications of $p$-Laplace equations in conformal geometry, we refer the reader to \cite{LMQZ2023b,LMQZ2023a} and \cite{MQ2021,MQ2022} for the cases $1<p<n$ and $p=n$, respectively.
	A comprehensive review of this field is beyond the scope of this paper. Instead, it serves to complement our previous results and provide a full classification of isolated singularities in this context. Throughout the paper, $\Omega\subset\mathbb{R}^{n}$ denotes a domain. Unless otherwise specified, $B_{r}$ denotes the open ball in $\mathbb{R}^{n}$ of radius $r$ centered at the origin, and $B_r(x)$ denotes the ball of radius $r$ centered at $x$. In
	\cite{MZ2025}, the authors classified the isolated singularities
	of positive solutions of 
	\begin{equation}
		0\le-\Delta_{p}u\le u^{\frac{n(p-1)}{n-p}},\,\, x\in B_1\setminus\{0\}\subset\mathbb{R}^{n},\,\,1<p<n.\label{Ma-Zang 2025 inequality}
	\end{equation}
	and obtained the following theorem.
	\begin{letterthm}[{\cite[Theorem 1.2]{MZ2025}}]\label{thm:MZ_old} Let $u$ be a $C^{2}$ positive
		solution of (\ref{Ma-Zang 2025 inequality}). Then $\lim_{x\to0}u(x)/\underline{u}(|x|)=1$. Moreover,
		exactly one of
		the following holds:
		\begin{enumerate}
			\item $u$ admits a $C^{1,\alpha}$ extension to $0$ for some $\alpha>0;$
			\item 
			\[
			\lim_{x\to0}|x|^{\frac{n-p}{p-1}}u(x)=m\in(0,\infty);
			\]
			\item 
			\[
			\liminf_{x\to0}\left(\log\frac{1}{|x|}\right)^{\frac{n-p}{p(p-1)}}|x|^{\frac{n-p}{p-1}}u(x)\ge\left(\left(\frac{n-p}{p}\right)\left(\frac{n-p}{p-1}\right)^{p-1}\right)^{\frac{n-p}{p(p-1)}}.
			\]
		\end{enumerate}
	\end{letterthm}
	
	A natural remaining question is to determine the precise asymptotic behavior  for the positive singular solution of the equation
	\begin{equation}
		-\Delta_{p}u=u^{\frac{n(p-1)}{n-p}},\,\, 1<p<n.\label{serrin exponent}
	\end{equation}
	Our first result is as follows.
	
	\begin{mainthm}\label{main thm1}
		
		Let $u\in C^{2}(B_1\setminus\{0\})$ be a positive solution
		of (\ref{serrin exponent}). Then either $u$ admits a $C^{1,\alpha}$
		extension to $\{0\}$ for some $\alpha>0$ or 
		\[
		\lim_{x\to0}\left(\log\frac{1}{|x|}\right)^{\frac{n-p}{p(p-1)}}|x|^{\frac{n-p}{p-1}}u(x)=\left(\left(\frac{n-p}{p}\right)\left(\frac{n-p}{p-1}\right)^{p-1}\right)^{\frac{n-p}{p(p-1)}}.
		\]
		
	\end{mainthm}
	
	This result establishes the $p$-Laplace analogue of \cite[Theorem A]{Av1987}. The corresponding result for radial solutions was previously obtained by Guedda and Veron \cite{GV1988}. In the classical case, the moving plane method in \cite{CGS1989} serves as a standard approach to deriving symmetry and asymptotic results. However, extending this technique to the $p$-Laplace equation \eqref{serrin exponent} is technically demanding due to the lack of linear structure and appropriate comparison principles. Adopting the strategy initiated in \cite{MZ2025}, we employ the Emden-Fowler transformation to convert the problem into an analysis of a dynamical system within the framework of viscosity theory, thereby circumventing the difficulties associated with the moving plane method in this context.

	Turning to the case $p=n$, a similar classification issue arises for the $n$-Laplace equation, where the fundamental solution behaves logarithmically. In \cite{MQ2021}, the authors studied the following equation
	\[
	-\Delta_{n}u=f(x,u,\nabla u)
	\]
	and obtained the following  result.
	\begin{letterthm}[{\cite[Theorem 4.2 and Remark 4.1]{MQ2021}}]\label{thm:MQ_old}
		Let $u\in C^{2}(B_2\setminus\{0\})$ be a nonnegative solution of 
		\begin{equation}
			-\Delta_{n}u=f(x,u,\nabla u)\label{n-Laplace equation}
		\end{equation}
		with 
		\[
		\lim_{x\to0}u(x)=\infty,
		\]
		where $f$ satisfies
		\begin{equation}
			0\le f(x,u,\nabla u)\le C|\nabla u|^{\beta}e^{\kappa u}\label{f condition}
		\end{equation}
		for $\beta\in (0,n)$ and $\kappa>0$. 
		Then 
		\[
		\lim_{x\to0}\frac{u(x)}{\log\frac{1}{|x|}}=m\ge0
		\]
		and 
		\[
		u(x)\ge m\log\frac{1}{|x|}
		\]
		for $x\in B_1\setminus\{0\}$.
	\end{letterthm}
	
	The case $m=0$ was not explicitly discussed in \cite{MQ2021}. Our second main result complements the previous theorem by addressing this borderline case.

	\begin{mainthm}\label{main thm 2}
		
		Let $w\in C^{2}(B_1\setminus\{0\})$ be a nonnegative solution
		of (\ref{n-Laplace equation}) with $f$ satisfying (\ref{f condition}).
		Then either $w$ admits a $C^{1,\alpha}$ extension to $\{0\}$ for some
		$0<\alpha<1$ or 
		\[
		\lim_{x\to0}\frac{w(x)}{\log\frac{1}{|x|}}=m>0.
		\]

	\end{mainthm}
	
	Combined with Theorem \ref{thm:MQ_old}, this result completes the description of the asymptotic behavior of solutions near the isolated singularity, covering both the removable singularity case ($m=0$) and the singular case ($m>0$). This fully generalizes Taliaferro's classification \cite[Theorem 2.1]{Ta2006} to the $n$-Laplace setting. The proof relies on potential theoretic arguments; specifically, we apply the pointwise potential estimates of Kilpel\"ainen and Mal\'y \cite{KM1994} and the local regularity theory of Serrin \cite{Se1964} to establish the removability of the singularity.
	
	\noindent\textbf{Note added.} The present manuscript was submitted for publication on May 8, 2026, before being posted on arXiv. While preparing this arXiv version, we became aware of the recent preprint by Xu and Yang \cite{XY2026}, posted on October 16, 2026, which establishes the same isolated-singularity classification for $1<p<n$ in the class of nonnegative weak solutions, whereas Theorem \ref{main thm1} is stated for positive $C^2$ solutions. Their proof builds on the spherical-extrema approach introduced in our earlier work \cite{MZ2025}, with an additional BV argument to handle weak solutions. The present paper also treats the case $p=n$ in Theorem \ref{main thm 2}, which is not covered in \cite{XY2026}.
	
	The rest of the paper is organized as follows. In Section 2, we collect some basic definitions and preliminary results. We then prove Theorem \ref{main thm1} and Theorem \ref{main thm 2} in Section 3 and Section 4 respectively.
	
	\section{Preliminaries}
	
	In this section, we collect several definitions and basic results that will be used throughout the paper.
	\begin{Def}[{\cite[Definition 2.5]{Li2017}}] Let $\Omega$ be a domain
		in $\mathbb{R}^{n}.$ We say that $u\in W_{loc}^{1,p}(\Omega)$ is
		a weak solution to 
		\[
		-\Delta_{p}u=0,
		\]
		if for each $\phi\in C_{0}^{\infty}(\Omega),$ there holds 
		\[
		\int_{\Omega}\langle|\nabla u|^{p-2}\nabla u,\nabla\phi\rangle dx=0.
		\]
		If in addition, $u$ is continuous, we say $u$ is $p$-harmonic.
	\end{Def}
	
	\begin{Def}[{\cite[Definition 5.1]{Li2017}}] A function $u:\Omega\to(-\infty, \infty]$
		is called $p$-superharmonic in $\Omega$, if 
		\begin{itemize}
			\item $u$ is lower semi-continuous in $\Omega;$
			\item $u\not\equiv\infty$ in $\Omega$;
			\item for each domain $D\Subset\Omega$ the comparison principle holds:
			if $h\in C(\bar{D})$ is $p$-harmonic in $D$ and $h|_{\partial D}\le u|_{\partial D},$
			then $h\le u$ in $D$. 
		\end{itemize}
	\end{Def}
	
	One crucial estimate for nonnegative $p$-superharmonic functions
	is the following theorem due to Kilpel\"ainen and Mal\'y \cite{KM1994}.
	
	\begin{thm}[{\cite[Theorem 1.6]{KM1994}}]\label{thm:KM94}
		Let $1<p\le n$. Suppose $u$
		is a nonnegative $p$-superharmonic function satisfying 
		\[
		-\Delta_{p}u=\mu,\,\, x\in B_{3r}(x_{0}).
		\]
		Then there exists $C_{1}(n,p),C_{2}(n,p)>0$ such that 
		\[
		C_{1}(n,p)W_{1,p}^{\mu}(x_{0},r)\le u(x_{0})\le C_{2}(n,p)\left(\inf_{B_{r}(x_{0})}u+W_{1,p}^{\mu}(x_{0},2r)\right).
		\]
		
	\end{thm}
	
	\section{Proof of Theorem \ref{main thm1}}
	
	The following lemma is a generalization of Lions \cite{Lio1980},  which is the $p=2$ case. 
	\begin{lem}\label{u L sigma}
		
		Let $u\ge0$ be a solution of
		\begin{equation}
			-\Delta_{p}u=u^{\sigma}\quad \text{in } B_1\setminus\{0\} \label{p Laplacian = u sigma}
		\end{equation}
		with $1<p<n$. If $\sigma\ge\frac{n(p-1)}{n-p}$, then $u^{\sigma}\in L^{1}(B_1)$
		and \eqref{p Laplacian = u sigma} holds in $B_1$ in the sense of distributions.
	\end{lem}
	\begin{proof}
		This follows directly from Phuc and Verbitsky \cite[Theorem 2.18]{PV2008}, which implies that \eqref{p Laplacian = u sigma} holds in $B_1$ if and only if
		\begin{equation*}
			\text{Cap}_{\mathbf{G}_{p},\frac{\sigma}{\sigma+1-p}}(\{0\})=0.
		\end{equation*}
		Noting that a point has zero Bessel capacity $\text{Cap}_{\mathbf{G}_\alpha, s}$ if and only if $\alpha s \le n$, this condition is equivalent to
		\[
		\frac{p\sigma}{\sigma+1-p} \le n, \quad \text{i.e.,} \quad \sigma\ge\frac{n(p-1)}{n-p}.
		\]
		Here, the Bessel capacity is defined for any set $E$  by
		\[
		\text{Cap}_{\mathbf{G}_{\alpha},s}(E) := \inf\big\{ \|f\|^s_{L^s(\mathbb{R}^n)}: \mathbf{G}_\alpha f\ge \chi_E, \, f\ge 0 \big\}.
		\]

	\end{proof}
	
	\begin{lem}Under the assumption of Theorem \ref{main thm1}, we have
		\[
		\lim_{x\to0}\frac{u(x)}{|x|^{\frac{p-n}{p-1}}}=0.
		\]
		
	\end{lem}
	
	\begin{proof}
		
		By  \cite[Theorem 1.2]{MZ2025}, there exists a constant $m\ge0$ such that 
		\[
		\lim_{x\to0}\frac{u(x)}{|x|^{\frac{p-n}{p-1}}}=m.
		\]
		Moreover, Lemma \ref{u L sigma} implies that
		$u \in L^{\frac{n(p-1)}{n-p}}(B_1),$
		and therefore $-\Delta_p u \in L^1_{\mathrm{loc}}(B_1)$.
		As a consequence, the Radon measure $-\Delta_p u$ has no mass concentration at the origin.
		It then follows from \cite[Theorem~4]{LMQZ2023b} that $m=0$.
	\end{proof}
	
	Then we follow the framework of \cite{MZ2025}. For $r>0$, we let
	\[
	\underline{u}(r)=\inf_{\partial B_{r}}u(x),\,\,\bar{u}(r)=\sup_{\partial B_{r}}u(x).
	\]
	
	\begin{Def}
		Let $f: (\alpha,\beta) \to \mathbb{R}$ be a function.
		\begin{enumerate}
			\item $f$ is called \textbf{concave} if for any $r_{1},r_{2}\in(\alpha,\beta)$ and $\lambda\in(0,1)$,
			\[
			f(\lambda r_{1}+(1-\lambda)r_{2})\ge\lambda f(r_{1})+(1-\lambda)f(r_{2});
			\]
			
			\item $f$ is called \textbf{semi-concave} if there exists a constant $C>0$ such that the function
			\[
			r \mapsto f(r)-Cr^{2}
			\]
			is concave in $(\alpha,\beta)$;
			
			\item $f$ is called \textbf{locally semi-concave} if for any $[\alpha_{1},\beta_{1}]\subset(\alpha,\beta)$, $f$ is semi-concave in $(\alpha_{1},\beta_{1})$.
		\end{enumerate}
	\end{Def}
	
	We recall the following properties of semi-concave functions.

	\begin{lem}\label{properties of semi concave} 
		Let $f: (\alpha,\beta) \to \mathbb{R}$ be a locally semi-concave function.
		\begin{enumerate}
			\item $f$ is continuous, and the one-sided derivatives $f_{r}^{\pm}$ exist everywhere. The derivative $f_{r}^{-}$ is left-continuous, while $f_{r}^{+}$ is right-continuous. Moreover, the strict inequality $f_{r}^{-}(r) > f_{r}^{+}(r)$ occurs only on an at most countable set.
			
			\item The distributional second derivative of $f$ is a Radon measure $\nu$, which admits the decomposition $\nu=\nu_{ac}+\nu_{pp}+\nu_{sc}$. Here, $\nu_{ac}$ denotes the absolutely continuous part with respect to the Lebesgue measure, $\nu_{pp}$ is the pure point component supported on $\{r_{i}\}$, and $\nu_{sc}$ is the singular continuous component. Furthermore, the singular parts are nonpositive, i.e., $\nu_{pp} \le 0$ and $\nu_{sc} \le 0$.
			
			\item The generalized Newton-Leibniz formula holds: for any $[\alpha_1, \beta_1] \subset (\alpha, \beta)$,
			\[
			f_{r}^{-}(\beta_{1})-f_{r}^{+}(\alpha_{1})=\int_{\alpha_1}^{\beta_1} d\nu.
			\]
		\end{enumerate}
	\end{lem}

	We now apply these properties to our specific setting via the Emden-Fowler transformation. Let $\eta=\frac{n-p}{p-1}$ and consider the change of variables $t=\eta\log\frac{\eta}{r}$ for $r\in(0,\eta)$. We define the function $\underline{w}(t)$ on $(\eta\log\eta,\infty)$ by $e^{t}\underline{w}(t)=\underline{u}(r)$.

	\begin{lem}\label{property of underline w}
		The function $\underline{w}(t)$ is locally semi-concave on $(\eta\log\eta,\infty)$ with $$\lim_{t\to\infty}\underline{w}(t)=0.$$ Moreover, in the decomposition $\underline{w}_{tt} = \nu_{ac} + \nu_{pp} + \nu_{sc}$, the singular parts $\nu_{pp}$ and $\nu_{sc}$ are nonpositive, and $\nu_{pp}$ is supported on a countable set $\{\tilde{t}_{i}\}$.
	\end{lem}

	The reader may refer to \cite[Lemma 6.3]{MZ2025} for the proof. To proceed with the analysis, we denote $\underline{z}(t)=\underline{w}_{t}^{+}(t)$ and $q=\frac{n(p-1)}{n-p}$, and consider the trajectory $(\underline{w}(t),\underline{z}(t))$ in the $(\underline{w},\underline{z})$-plane.
	
	\begin{lem}\label{property of underline w 2}
		\begin{enumerate}
			\item $\underline{w}(t)$ is continuous for $t\ge\eta\log\eta$. For $t\in \operatorname{supp}(\nu_{pp})=\{\tilde{t}_{i}\}$,
			$\underline{w}_{t}^{-}(\tilde{t}_{i})>\underline{w}_{t}^{+}(\tilde{t}_{i})$. Moreover,
			$\underline{z}(t)$ is lower semi-continuous, right continuous, of
			locally bounded variation in $[\eta\log\eta,\infty)$ and is continuous
			for $t\ne\tilde{t}_{i}$.
			\item If $\underline{w}(t)+\underline{z}(t)>0$ for $t\in[t_{0},t_{1}]\subset(\eta\log\eta,\infty)$
			then 
			\[
			\underline{z}(t_{1})-\underline{z}(t_{0})\le\underline{w}_{t}^{-}(t_{1})-\underline{w}_{t}^{+}(t_{0})\le\int_{t_{0}}^{t_{1}}\left(-\underline{z}-\frac{\underline{w}^{q}}{(p-1)(\underline{w}+\underline{z})^{p-2}}\right)dt.
			\]
		\end{enumerate}
	\end{lem}

	The proof is a direct consequence of Lemma \ref{properties of semi concave} and is analogous to that of \cite[Lemma 6.4]{MZ2025}.
	
	\begin{lem}\label{z less than 0}
		
		There is $T>0$ such that, $\underline{z}(t)<0$ for all $t\in[T,\infty)$.
	\end{lem}
	
	\begin{proof}
		
		Choose $T>\eta\log\eta$ such that $\underline z(T)<0$.
		Suppose, for contradiction, that $\underline z(t)<0$ does not hold for all $t\ge T$.
		Set
		\[
		t_{1}=\sup\left\{ s\ge T:\;\underline{z}(t)<0\text{ for all } t\in [T,s)\right\} .
		\]
		By definition we have $\underline{z}(t)<0$ on $[T,t_{1})$.
		Since $\underline{z}(t)$
		is right continuous, we must have $\underline{z}(t_{1})\ge0$. On the other hand, the lower semicontinuity of $\underline z$ together with $\underline z(t)<0$ for $t<t_1$ implies
		\[
		\lim_{t\to t_1^-}\underline z(t)=0,
		\]
		hence $\underline z(t_1)=0$. Moreover, since 
		\[
		0\ge\liminf_{t\to t_{1}}\underline{z}(t)\ge\underline{z}(t_{1})=0,
		\]
		we derive 
		\[
		\lim_{t\to t_{1}}\underline{z}(t)=0.
		\]
		Next, observe that $w(t)$ is monotone decreasing on $[T,t_1]$, so $w(t)\ge w(t_1)$ for $t\in[T,t_1]$.
		Therefore we can choose $\delta>0$ sufficiently small so that on $[t_1-\delta,t_1]$ the following strict inequality holds
		\[
		-\underline{z}-\frac{\underline{w}^{q}}{(p-1)(\underline{w}+\underline{z})^{p-2}}<0.
		\]
		By Item 2 of Lemma \ref{property of underline w 2}, $\underline{z}(t_{1})<\underline{z}(t_{1}-\delta)$,
		which is a contradiction.
	\end{proof}
	
	For fixed $0<\varepsilon_{0},\delta_{0}<1$, we let $k=k(\varepsilon_{0})=\frac{1}{(p-1)(1+\varepsilon_{0})}$,
	and 
	\[
	\Theta_{k}^{\delta_{0}}=\left\{ (\underline{w},\underline{z});0<\underline{w}\le\delta_{0},\underline{z}<-k\underline{w}^{q-p+2}\right\} .
	\]

	\begin{lem}
		Suppose that for fixed $\varepsilon_{0}, \delta_{0} \in (0,1)$, there exists $T>0$ such that
		\[
		(\underline{w}(t),\underline{z}(t))\in\Theta_{k}^{\delta_{0}} \quad \text{for all } t>T.
		\]
		Then the conclusion of Theorem \ref{main thm1} holds.
	\end{lem}
	
	\begin{proof}
		
		Firstly, we establish  that if the trajectory $(\underline{w}(t),\underline{z}(t))$ remains in the region $\Theta_{k}^{\delta_{0}}$ for all $t>T$,
		then the differential inequality
		\[
		\frac{d}{dt}\underline{w}(t)<-k\underline{w}^{q-p+2}
		\]
		holds in the sense of distributions on $(T,\infty)$.
		Indeed, by Lemma \ref{property of underline w 2}, $\underline{w}_{t}^{-}(\tilde{t}_{i})>\underline{w}_{t}^{+}(\tilde{t}_{i})$
		on $\{\tilde{t}_{i}\}=\operatorname{supp}\left(\nu_{pp}\right)$
		and we obtain
		\[
		\frac{d}{dt}\underline{w}^{p-q-1}>-k(p-q-1)
		\]
		in the sense of distributions. For fixed $t_{1}>T$, integrating the inequality above from $t_1$ to $t>t_1$ gives
		\[
		\underline{w}^{p-q-1}(t)-\underline{w}^{p-q-1}(t_{1})>-k(p-q-1)(t-t_{1}),
		\]
		and hence, for $t$ large, 
		\begin{align*}
			\underline{w}^{p-q-1}(t) & \ge-\frac{p-q-1}{(p-1)(1+\varepsilon_{0})}t+C\\
			& =\frac{p}{(n-p)(1+\varepsilon_{0})}t+C,\\
			& \ge\frac{p}{(n-p)(1+2\varepsilon_{0})}t.
		\end{align*}
		Consequently,  for all sufficiently large $t$, 
		\[
		\underline{w}(t)\le\left(\frac{p}{(n-p)(1+2\varepsilon_{0})}\right)^{-\frac{n-p}{p(p-1)}}t^{-\frac{n-p}{p(p-1)}}.
		\]
		Since $\varepsilon_0>0$ was arbitrary (and may be taken arbitrarily small), we deduce
		\[
		\limsup_{t\to\infty}t^{\frac{n-p}{p(p-1)}}\underline{w}(t)\le\left(\frac{p}{n-p}\right)^{-\frac{n-p}{p(p-1)}}.
		\]
		Translating this estimate back to the original variable $x$ yields
		\[
		\limsup_{x\to0}\left(\log\frac{1}{|x|}\right)^{\frac{n-p}{p(p-1)}}|x|^{\frac{n-p}{p-1}}\underline{u}(|x|)\le\left(\left(\frac{n-p}{p}\right)\left(\frac{n-p}{p-1}\right)^{p-1}\right)^{\frac{n-p}{p(p-1)}}.
		\]
		
		On the other hand, if the singularity is non-removable, Theorem \ref{thm:MZ_old} provides the matching lower bound
		\[
		\liminf_{x\to0}\left(\log\frac{1}{|x|}\right)^{\frac{n-p}{p(p-1)}}|x|^{\frac{n-p}{p-1}}\bar{u}(|x|)\ge\left(\left(\frac{n-p}{p}\right)\left(\frac{n-p}{p-1}\right)^{p-1}\right)^{\frac{n-p}{p(p-1)}}
		\]
		and moreover $\bar u(|x|)/\underline u(|x|)\to1$ as $x\to0$. Combining these facts yields
		\[
		\lim_{x\to0}\left(\log\frac{1}{|x|}\right)^{\frac{n-p}{p(p-1)}}|x|^{\frac{n-p}{p-1}}u(x)=\left(\left(\frac{n-p}{p}\right)\left(\frac{n-p}{p-1}\right)^{p-1}\right)^{\frac{n-p}{p(p-1)}},
		\]
		which is precisely the conclusion of Theorem~\ref{main thm1}.
	\end{proof}
	
	Central to our argument is the following lemma.
	
	\begin{lem}\label{key lemma to prove thm 1}
		There exists $T>0$ such that if $(\underline{w}(t_{0}),\underline{z}(t_{0}))\in\Theta_{k}^{\delta_{0}}$ for some $t_{0}>T$, then the trajectory remains in $\Theta_{k}^{\delta_{0}}$ for all $t>t_{0}$.
	\end{lem}
	
	\begin{proof}
		The proof is analogous to that of \cite[Lemma 6.6]{MZ2025}. Yet we still provide the full details below. 
		Since 
		\[
		\lim_{t\to\infty}\underline{w}(t)=0,
		\]
		we may choose $T>0$ sufficiently large such that $0<\underline{w}(t)<\delta_{0}$ for all $t\ge T$. Assume now that for some $t_{0}\ge T$, 
		\[
		(\underline{w}(t_{0}),\underline{z}(t_{0}))\in\Theta_{k}^{\delta_{0}}.
		\]
		For $t\ge T$, we have $\underline{w}(t)>0$ and $\underline{z}(t)$ is locally bounded from below. Therefore, the only possible way in which the trajectory $(\underline{w}(t),\underline{z}(t))$ could leave $\Theta_{k}^{\delta_0}$ for $t>t_{0}$ is that there exists $t_{1}>t_{0}$ such that $\underline{z}(t_{1})=-k\underline{w}^{q-p+2}(t_{1})$ and
		\[
		(\underline{w}(t),\underline{z}(t))\in\Theta_{k}^{\delta_{0}}
		\qquad\text{for all } t\in[t_{0},t_{1}),
		\]
		and
		\[
		(\underline{w}(t),\underline{z}(t))\to(\underline{w}(t_{1}),\underline{z}(t_{1}))
		\quad\text{as }t\to t_{1}^{-}.
		\]
		From this we infer 
		\begin{align*}
			\limsup_{t\to t_{1}^{-}}\frac{\underline{z}(t)-\underline{z}(t_{1})}{\underline{w}(t)-\underline{w}(t_{1})} & \le\frac{d(-k\underline{w}^{q-p+2})}{d\underline{w}}\Big|_{\underline{w}=\underline{w}(t_{1})}\\
			& =-k(q-p+2)\underline{w}^{q-p+1}(t_{1})\\
			& \to0
		\end{align*}
		as $t_{1}\to\infty$.
		
		On the other hand, since $\underline{z}(t)<0$ for $t\in[t_{0},t_{1}]$, we have $\underline{w}(t_{1})<\underline{w}(t)$ for $t<t_{1}$. Using Lemma \ref{property of underline w 2}, we obtain
		\begin{align*}
			\liminf_{t\to t_{1}^{-}}\frac{\underline{z}(t_{1})-\underline{z}(t)}{\underline{w}(t_{1})-\underline{w}(t)} & =\liminf_{t\to t_{1}^{-}}\frac{\underline{z}(t_{1})-\underline{z}(t)}{t_{1}-t}\frac{t_{1}-t}{\underline{w}(t_{1})-\underline{w}(t)}\\
			& \ge\liminf_{t\to t_{1}^{-}}\frac{\int_{t}^{t_{1}}\left(-\underline{z}-\frac{\underline{w}^{q}}{(p-1)|\underline{w}+\underline{z}|^{p-2}}\right)d\tau}{t_{1}-t}\frac{t_{1}-t}{\underline{w}(t_{1})-\underline{w}(t)}\\
			& =\left(-\underline{z}(t_{1})-\frac{\underline{w}^{q}(t_{1})}{(p-1)|\underline{w}(t_{1})+\underline{z}(t_{1})|^{p-2}}\right)\frac{1}{\underline{z}(t_{1})}\\
			& =-1+\frac{\underline{w}^{p-2}(t_{1})}{k(p-1)|\underline{w}(t_{1})+\underline{z}(t_{1})|^{p-2}}\\
			& \to-1+1+\varepsilon_{0}=\varepsilon_{0}
		\end{align*}
		as $t_{1}\to\infty.$ 
		
		For $t_{1}$ sufficiently large, the above lower bound $\varepsilon_{0}>0$ contradicts the previous upper bound tending to $0$, which ends the proof of Lemma \ref{key lemma to prove thm 1}.
	\end{proof}

	In light of Lemma \ref{key lemma to prove thm 1}, the proof of Theorem \ref{main thm1} is reduced to establishing the existence of an entry time into the region $\Theta_{k}^{\delta_{0}}$.
	
	\begin{lem}
		Let $\varepsilon_{0}, \delta_{0}, k$ and $T$ be as defined in Lemma \ref{key lemma to prove thm 1}. Then there exists $t_{0}\ge T$ such that $(\underline{w}(t_{0}),\underline{z}(t_{0}))\in\Theta_{k}^{\delta_{0}}.$
	\end{lem}
	
	\begin{proof}
		
		Suppose by contradiction that no such $t_0$ exists; equivalently, $$\underline{z}(t)\ge-k\underline{w}(t)^{q-p+2} \text{ for all } t\ge T.$$
		Since $\underline{z}<0$ and $\underline{w}$ is monotonically
		decreasing in $t$, we may regard $\underline{z}$ as a function of
		$\underline{w}$.
		In particular
		\begin{equation}
			\lim_{\underline{w}\to0}\underline{z}(\underline{w})=0.\label{z go to 0}
		\end{equation}
		
		However, on the other hand
		
		\[
		\frac{d\underline{z}}{dt}\le-\underline{z}-\frac{\underline{w}^{q}}{(p-1)(\underline{w}+\underline{z})^{p-2}}
		\]
		in the sense of distributions. 
		And if $t\ne\tilde{t}_{i}$, 
		\[
		\frac{d\underline{w}}{dt}=\underline{z}<0.
		\]
		For such $t\ne\tilde{t}_{i}$, since $\underline{z}\ge-k\underline{w}^{q-p+2}$, we obtain
		\begin{align*}
			\frac{d\underline{z}}{d\underline{w}} & \ge-1-\frac{\underline{w}^{q}}{(p-1)(\underline{w}+\underline{z})^{p-2}\underline{z}}\\
			& \ge-1-\frac{\underline{w}^{p-2}}{k(p-1)(\underline{w}+\underline{z})^{p-2}}\\
			& \ge-1-\frac{1+\varepsilon}{(1-k\underline{w}^{q-p+1})^{p-2}}\\
			& \ge\frac{\varepsilon}{2},
		\end{align*}
		for $t>0$ large enough.
		If $t\ne\tilde{t}_{i}$, since 
		\[
		\lim_{t\to\tilde{t}_{i}^{-}}\underline{z}(t)>\underline{z}(\tilde{t}_{i}),
		\]
		it follows that
		\[
		\frac{d\underline{z}}{d\underline{w}}|_{t=\tilde{t}_{i}}=m_{i}\delta(t-\tilde{t}_{i})
		\]
		for $m_{i}>0$, where $\delta$ denotes the Dirac measure. Consequently, we have
		\[
		\frac{d\underline{z}}{d\underline{w}}\ge\frac{\varepsilon}{2}
		\]
		in the sense of distributions for sufficiently large $t$.
		This contradicts  the fact that $\underline{z}(\underline{w})<0$. 
		
	\end{proof}
	
	\section{Proof of Theorem \ref{main thm 2}}
	
	Recall that by \cite[Proposition 1.1]{BV1989}, we have $w\in L_{loc}^{s}(B_{1})$ for all $s\in[1,\infty)$ and $\nabla w\in L_{loc}^{q}(B_{1})$ for all $q\in[1,n)$. Moreover, if $\lim_{x\to0}w(x)=\infty,$ then $w$ satisfies
	\[
	-\Delta_{n}w=\tau\delta(x)+g \quad \text{in } \mathcal{D}'(B_1),
	\]
	where $\tau\ge0$ and $g\in L^{1}(B_{1})$ is nonnegative. By Theorem \ref{thm:MQ_old}, we also have
	\[
	\lim_{x\to0}\frac{w(x)}{\log\frac{1}{|x|}}=m\ge0.
	\]
	
	It remains to analyze the case $m=0$. We claim that $m=0$ implies $\tau=0$. Indeed, if $\tau>0$, then by Theorem \ref{thm:KM94}, there exists a constant $C>0$ such that
	\[
	w(x)\ge C\,W_{1,n}^{\tau\delta(x)+g}\!\big(x,\tfrac{1}{10}\big)\qquad\text{for }x\in B_{1/2}.
	\]
	According to \cite[Theorem~1]{LMQZ2023b}, this lower bound implies
	\[
	\liminf_{x\to 0}\frac{w(x)}{\log\frac{1}{|x|}}>0,
	\]
	which contradicts the assumption $m=0$. Thus, we must have $\tau=0$.
	
	Since $m=0$, for any $\varepsilon>0$, there exists $r_{\varepsilon}>0$ such that
	\[
	0\le w(x)\le \varepsilon\log\frac{1}{|x|}\qquad\text{for }0<|x|<r_\varepsilon.
	\]
	This logarithmic growth bound implies that $e^{\kappa w}\in L_{loc}^{s}(B_{1})$ for every $s\in[1,\infty)$. Combining this integrability with the gradient estimate in \cite[Theorem 5.17]{Li2017}, we deduce that
	\[
	|\nabla w|^{\beta}e^{\kappa w}\in L^{\gamma}(B_{1}) \quad \textrm{for some }\gamma>1.
	\]
	Consequently, the structure conditions (5), (6), and (8) in \cite{Se1964} are satisfied, and we may invoke \cite[Theorem~8]{Se1964} to conclude that $w$ is H{\"o}lder continuous in $B_1$. Finally, by \cite[Theorem~2]{Di1983} (or \cite[Theorem~1]{To1984}), it follows that $\nabla w$ is also H{\"o}lder continuous in $B_1$. This completes the proof.

	\vspace{1\baselineskip}
	
	\noindent \textbf{Acknowledgements:} The first named author is supported by Natural Science Foundation of Tianjin (Grant No. 22JCJQJC00130), the Fundamental Research Funds for the Central Universities.

	\bigskip
	
	\noindent Shiguang Ma: School of Mathematical Sciences and LPMC, Nankai University\\ 
	Tianjin 300071, China \\
	\textit{e-mail}: msgdyx8741@nankai.edu.cn \\
	
	\medskip
	
	\noindent Shengyang Zang: School of Mathematical Sciences, Nankai University\\ 
	Tianjin 300071, China\\
	\textit{e-mail}: 1120230032@mail.nankai.edu.cn

\end{document}